\documentclass[12pt]{article}

\usepackage{amsmath,amsthm,amsfonts,latexsym,amscd,amssymb}
\input xypic
\def\qed{{\hbadness=10000\hfill\ \vbox{\hrule height.09ex
   \hbox{\vrule width.09ex height1.55ex depth.2ex \kern1.8ex
   \vrule width.09ex height1.55ex depth.2ex}\hrule height.09ex}\break
   \bigskip}}
\newtheorem{theorem}{Theorem}[section]
\newtheorem{lemma}[theorem]{Lemma}
\newtheorem{corollary}[theorem]{Corollary}
\newtheorem{proposition}[theorem]{Proposition}

\theoremstyle{definition}

\newtheorem{example}[theorem]{Example}

\theoremstyle{remark}
\newtheorem{remark}[theorem]{Remark}

\newcommand{\N}{\mathbb{N}}

\newcommand{\Z}{\mathbb{Z}}

\newcommand{\n}{\noindent}
\begin{document}

\title{Finite Just Non-Dedekind Groups}

\author{V.K. Jain\footnote{supported by UGC, Government of India}
 ~and R.P. Shukla\\
Department of Mathematics, University of Allahabad \\
Allahabad (India) 211 002\\
{\bf Email:} jaijinenedra@gmail.com; rps@mri.ernet.in}

\date{}
\maketitle

\n {\bf Running Title:} Finite Just Non-Dedekind Groups.

\vspace{2cm}

\n \textbf{Abstract:} A group is just non-Dedekind (JND) if it is not a Dedekind group but all of whose proper homomorphic images are Dedekind groups. The aim of the paper is to classify finite JND-groups.

\newpage

\n \textbf{Key-words:} Monolith, Characteristically simple, Semisimple.

\smallskip

\n \textbf{MSC 2000:} 20E99, 20D99
\section{Introduction}

A group is called Dedekind if all its subgroups are normal.
 By \cite{b1}, a group is Dedekind if and only if it is abelian or the direct product of a quaternion group of order $8$, an elementary abelian $2$-group and an abelian group with all its elements of odd order (one can also see its proof in \cite[5.3.7, p.143]{rob}).

Given a group theoretical property ${\cal P}$, a just non ${\cal P}$-group is a group which is not ${\cal P}$-group but all of whose proper homomorphic images are ${\cal P}$-groups; for brevity we shall call these JN${\cal P}$-groups. 
M.F. Newman studied just nonabelian (JNA) groups in \cite{new1,new2}. S. Franciosi and others studied solvable just nonnilpotent (JNN) groups in \cite{fs} and D.J.S. Robinson studied solvable just non-T (JNT) groups in     \cite{rob73}(here the group with property T means in the group normality is a transitive relation).

The aim of this paper is to classify finite JND-groups. In Section $2$, we prove that JND-groups are monolithic group. Section 3 deals with solvable JND-groups and Section 4 shows that nonsolvable JND-groups are semisimple. Theorem \ref{th1} gives complete classification of finite semisimple JND-groups.

Let $G$ be a group. For, sets $X$, $Y$ of $G$, let $[X,Y]$ denote the subgroup of $G$ generated by $[x,y]=xyx^{-1}y^{-1},~x\in X,~y\in Y$.
The derived series of $G$ is $$G=G^{(0)} \geq G^{(1)} \geq \cdots \geq G^{(n)} \geq \cdots ,$$ where $G^{(n)}=[G^{(n-1)},G^{(n-1)}]$, the commutator subgroup of $G^{(n-1)}$. The lower central series of $G$ is $$G=\gamma_1(G) \geq \gamma_2(G) \geq \cdots \geq \gamma_n(G) \geq \cdots ,$$ where $\gamma_{n+1}(G)=[\gamma_n(G),G]$.
The group $G$ is called {\it solvable} of {\it derived length} $n$ (respectively {\it nilpotent} of {\it class $n$}) if $n$ is the smallest nonnegative integer such that $G^{(n)}=\{1\}$ (respectively $\gamma_{n+1}(G)=\{1\})$.

\section{Some basic properties of JND-groups}
We recall that a group is called {\em{monolithic}} if it has smallest nontrivial normal subgroup, called the {\em{monolith}} of $G$.  
In this section, we study some basic properties of JND-groups.
\smallskip  
 
\begin{proposition} \label{d1}
Let $G$ be a JND-group.
 Then $G$ is not contained in a direct product of Dedekind groups. 
\end{proposition}
\begin{proof} 
Let $\{ H_i\}_{i\in I}$ denote a family of Dedekind groups, where $I$ is an indexing set. 
Assume that $G$ is contained in $H=\prod_{j\in I} H_j$. Since $G$ is nonabelian, there exists $i\in I$ such that $H_i$ is nonabelian. 
By the classification Theorem for nonabelian Dedekind groups \cite[5.3.7, p.143]{rob}, square of each element of a nonabelian Dedekind group is central and its commutator subgroup is isomorphic to $\Z_2$.
 This implies that $G$ can not be simple.
  Take any nontrivial element $x \in G$.
   Then $x\in Z(G)$ if $x^2=1$ and $x^2 \in Z(G)$ if $x^2 \neq 1$ (for $G$ is contained in $H$).
    This proves that each subgroup of $G$ contains a nontrivial central element of $H$. 
    Let $N$ be a nontrivial subgroup of $G$.
     Let $x\in N \cap Z(H)$, $x \neq	1$.
      Since $G$ is JND, $G/\langle x \rangle$ is Dedekind, so $N/\langle x \rangle \trianglelefteq G/\langle x \rangle$, which proves that $N\trianglelefteq G$.
 Hence $G$ is Dedekind.
\end{proof} 

\smallskip
\begin{corollary} \label{d2}
Let $G$ be a JND-group. Then $G$ is monolithic.  
\end{corollary}
\begin{proof}
If $G$ is a JNA-group, there is nothing to prove for $G^{(1)}$ will be contained in each nontrival normal subgroup of $G$. 
Assume that $G$ is not JNA.
 Let ${\cal A}$ denote the set of all nontrivial normal subgroups of $G$.
  Then $G/H$ is Dedekind for all $H\in {\cal A}$.
   Further, since $G$ is not JNA, there exists $H\in {\cal A}$ such that $G/H$ is nonabelian.
    Therefore by Proposition \ref{d1}, the homomorphism from $G$ to $\prod_{H\in {\cal A}} G/H$ which sends $x\in G$ to $(xH)_{H\in {\cal A}}$ is not one-one.
     This proves that $\bigcap _{H\in {\cal A}} H \neq \{ 1\}$.   
\end{proof}
\smallskip
\begin{corollary} \label{da}
Let $G$ be as in Corollary \ref{d2}. 
Assume that $G^{(2)} \neq \{1\}$.
 Then the monolith of $G$ is $G^{(2)}$. 
\end{corollary}
\begin{proof}
By Corollary \ref{d2}, $G$ is monolithic.
 Let $K$ denote the monolith of $G$.     
 Then $K \subseteq G^{(2)}$. If $G$ is JNA, then $K=G^{(1)}$ and so $K=G^{(2)}$.
  If $G$ is JND but not JNA, then $G/K$ is nonabelian Dedekind.
   Now by \cite[5.3.7, p.143]{rob}, the commutator subgroup $G^{(1)}/K$ of $G/K$ is of order 2. 
   So $G^{(2)} \subseteq K$.
\end{proof}

\section{Finite solvable JND-groups}

In this section, we classify finite solvable JND-groups. Solvable JNA-groups with  nontrival center is characterized in \cite{new2} and centerless solvable JNA-groups have been classified in  \cite[Theorem 5.2, p.360]{new1}. So, it only remains to classify finite solvable JND-groups which are not JNA-groups.
\smallskip
\begin{proposition}\label{d3} 
Let $G$ be a JND-group. 
Let $Z(G)$, the center of $G$ be nontrivial.
 Then $G$ is a solvable JNA-group.
\end{proposition}
\begin{proof}
Suppose that $G$ is JND but not JNA.
 By Corollary \ref{d2}, $G$ is monolithic. 
 Let $K$ denote the monolith of $G$. 
 Since every subgroup of $Z(G)$ is normal subgroup of $G$, $K$ is central subgroup of order $p$ for some prime $p$. 

\n We claim that $p=2$.
 Since $G$ is JND but not JNA, $G/K$ is nonabelian Dedekind.
  By the structure Theorem for nonabelian Dedekind groups 
  \cite[5.3.7, p.143]{rob}, the commutator $(G/K)^{(1)}=G^{(1)}/K$ is of order 2. 
  Thus $|G^{(1)}|=2p$. 
  Let $x$ be an element of $G^{(1)}$ of order 2. 
  If $x\in Z(G)$, then $K=\langle x \rangle $, so $p=2$. 
  Assume that $x \not \in Z(G)$. Since $|G^{(1)}/K|=2$ and $x\not \in K$, so $G^{(1)}=\langle x \rangle K$. Let $g\in G$ such that $gxg^{-1} \neq x$. Then there exists a nontrivial element $h\in K$ such that $gxg^{-1}=xh$. Now since $h \in Z(G)$, $h^2=x^2h^2=(xh)^2=(gxg^{-1})^2=1$ implies that $p=2$. 

Next, we show that $G$ does not contain an element of odd prime order. Assume that $x\in G$ is of odd prime order $q$.  Since $\langle x \rangle K$ has a unique subgroup of order $q$ and  $\langle x \rangle K \trianglelefteq G$ (for $G/K$ is Dedekind), $\langle x \rangle  \trianglelefteq G$. But, then $K\subseteq \langle x \rangle $, a contradiction. 

Further, since $G/K$ is a nonabelian Dedekind, by \cite[5.3.7, p.143]{rob}, $G$ does not contain any element of infinite order. Thus we have shown that $G$ is a $2$-group. Lastly, since $G/K$ is a nonabelian Dedekind, by \cite[5.3.7, p.143]{rob}, $G$ contains a nonabelian subgroup $H$ of order 16 such that $K \subseteq H$ and $H/K \cong Q_8$. But this is not possible \cite[118, p.146]{burn}.    
\end{proof}
\smallskip
\begin{lemma} \label{ps}
A finite centerless solvable JND-group is a JNT-group.
\end{lemma}
\begin{proof}
Let $G$ be a finite centerless solvable JND-group. Since a Dedekind group is also a T-group, it is sufficient to show that $G$ is not a $T$-group.

Suppose that $G$ is a $T$-group. Let $K$ denote the monolith of $G$ (Corollary \ref{d2}). Since $G$ is a finite solvable $T$-group, $K$ is a cyclic group of order $p$ for some prime $p$. Since $G/K$ is nonabelian Dedekind group, by \cite[5.3.7, p.143]{rob}$, |G^{(1)}/K|=2$. Further, since a solvable $T$-group is of derived length at most two \cite[13.4.2, p.403]{rob}, $G^{(1)}$ is abelian. Now since $G^{(1)}$ is an abelian group of order $2p$ and $G$ is a $T$-group, $p=2$. But then $K \subseteq Z(G)=\{1\}$. This is a contradiction. Therefore $G$ is a JNT-group.
\end{proof}

\smallskip
The following example shows that there exists a solvable JND-group which is not a JNA-group.
\smallskip
\begin{example}
Consider an elementary abelian $3$-group $A$ of order $9$. Let $\psi$ denote the homomorphism from $Q_8$  to $Aut ~A=Gl_2(3)$ defined as $\imath \mapsto \left(\begin{array}{cc}
0 & 1 \\
-1 & 0
\end{array}\right)
$, $\jmath  \mapsto \left(\begin{array}{cc}
1 & 1 \\
1 & -1
\end{array}\right)$, where $Q_8= \{\pm 1, ~\pm\imath, ~\pm \jmath,~ \pm k \}$ is the quaternion group of order $8$. It is easy to check that $\psi$ is injective. 
Let $G=AQ$ denote the natural semidirect product of $A$ by $Q_8$. 
Then $G$ is a JND-group with monolith $A$.
\end{example}
\smallskip
The following proposition classifies all finite solvable JND-groups which are not JNA-groups.
\smallskip
\begin{lemma}\label{ps1}
A finite solvable group $G$ is JND but not JNA if and only if there exists an elementary abelian normal $p$-subgroup $A$ of $G$ for some prime $p$ which is also monolith of $G$ and a nonabelian Dedekind group $X$ of $G$ such that $A \cap X=\{1\},~ G=AX$ and the conjugation action of $X$ on $A$ is faithful and irreducible. 
\end{lemma}
\begin{proof}
Suppose that $G$ is a finite solvable JND-group but not JNA-group. By Corollary \ref{d2}, $G$ is monolithic. Let $K$ be the monolith of $G$. Then $G/K$ is a nonabelian Dedekind. Thus by 
\cite[5.3.7, p.143]{rob}, $|G^{(1)}/K|=2$. Since $K$ is characteristically simple and abelian, it is an elementary abelian $p$-group of order $p^n$ for some prime $p$ \cite[3.3.15 (ii), p.87]{rob}.

Assume that $G^{(1)}$ is abelian. If $p \neq 2$, then $G^{(1)}$ contain unique element of order $2$ and so $Z(G) \neq 1$. By Proposition \ref{d3}, this is a contradiction. Thus $p=2$.
Now by Proposition \ref{d3}, $G$ is not nilpotent and by Lemma \ref{ps}, $G$ is a JNT-group. So by Case 6.2 and its Subcases 6.211, 6.212, 6.22 and 6.222 in \cite[pp.202-208]{rob73}, there is no finite JNT-group which is not JNA and has a minimal normal subgroup isomorphic to an elementary abelian $2$-group.

Assume that $G^{(1)}$ is nonabelian. Then $[G^{(1)}, K] \neq 1$, for $|G^{(1)}/K|=2$. Now since $G$ is a finite nonnilpotent JNT-group and $[G^{(1)}, K] \neq 1$, by Case 6.1
of \cite[p.202]{rob73}, there is a nontrival normal subgroup $A$ of $G$, a solvable $T$-subgroup $X$ of $G$ such that $A \cap X=\{1\}, ~G=AX$ and the conjugation action of $X$ on $A$ is faithful and irreducible. Further, since $K \subseteq A$ and the conjugation action of $X$ on $A$ is irreducible, $K = A$. So $X\cong G/A = G/K$ is a nonabelian Dedekind group.  

Conversely, suppose that $G=AX$, $A\cap X=\{1\}$, $X$ is a nonabelian Dedekind subgroup, $A$ is an elementary abelian $p$-group and also the monolith of $G$. Since $A$ is solvable and $G/A \cong X$ is nonabelian Dedekind and so solvable,  $G$ is solvable. Further, since $A$ is the monolith of $G$ and $G/A\cong X$ is nonabelian Dedekind, $G$ is JND but not JNA.
\end{proof}
\smallskip
The following proposition lists some more properties of finite solvable JND-groups which are not JNA-groups.
\smallskip
\begin{proposition}\label{c1}
Let $G$, $A$ and $X$ be as in the Lemma \ref{ps1}. Then

\noindent (i) The stabilizer of any nontrival element of $A$ is trival.

\noindent (ii) $|X|$ divides $p^n-1$, in particular $p$ and $|X|$ are coprime.

\noindent (iii) $X \cong Q_8 \times A_o$, where $A_o$ is a cyclic group of odd order.  

\end{proposition}

\begin{proof}
Let $a \in A,~ a \neq 1$. Assume that the stabilizer $stab_X(a)$ of $a$ in $X$ is nontrival. Assume that $x \in stab_X(a),~x \neq 1$. Since $G/A$ is a Dedekind group, $\langle x \rangle A \trianglelefteq G$. Thus $Z(\langle x \rangle A)$ is a nontrival normal subgroup of $G$ (for $a \in Z(\langle x \rangle A))$ and so $A \subseteq Z(\langle x \rangle A)$, for $A$ is the monolith of $G$. But this is a conradiction, for the conjugation action of $X$ on $A$ is faithful. This proves (i). Now (ii) is implied by the class equation for the action of $X$ on $A$.

Further, by \cite[Lemma 1, p.185]{rob73}, there is an extension field $E$ of $\Z_p$ such that $Z(X) \cong Y \leq E^{\star}$ and $E=\Z_p(Y)$, where $E^{\star}$ denote the multiplicative group of $E$. Clearly $E$ is a finite field , so $E^{\star}$ is a cyclic group. This implies $X\cong Q_8 \times A_o$, where $A_o$ is a cyclic group of odd order \cite[5.3.7, p.143]{rob}. This proves (iii).  \end{proof}

\section{Finite nonsolvable JND-groups}

Recall that a group is {\em{semisimple}} \cite[p.89]{rob} if its maximal solvable normal subgroup is trivial. Also a maximal normal completely reducible subgroup is called the {\em{CR-radical}} \cite[p.89]{rob}.
\smallskip
\begin{proposition}\label{d4}
Let $G$ be a finite nonsolvable JND-group. 
Then $G$ is a semisimple group.
\end{proposition}
\begin{proof}
Assume that $G$ has a nontrival normal solvable subgroup $N$. Then $G/N$ is a Dedekind group. Hence by \cite[5.3.7, p.143]{rob}, $G/N$ is solvable. But then $G$ is solvable, a contradiction. 
\end{proof}
\smallskip
  
Now we fix some notations for the rest of the section.
For a group $G$, we 
denote {\em{Inn $G$}} for the inner automorphism subgroup 
of {\em{Aut $G$}}, the automorphism group of $G$ and {\em{Out $G$}} 
for the outer automorphism group of $G$. Let $H$ denote a finite nonabelian simple group. 
Consider the semidirect product $\underbrace{(Aut~H \times \ldots \times Aut~H)}_{r~copies}\rtimes S_r$ and $\underbrace{(Out~H \times \ldots \times Out~H)}_{r~copies} \rtimes S_r$, where $S_r$ acts on $\underbrace{(Aut~H \times \ldots \times Aut~H)}_{r~copies}$ as well as on $\underbrace{(Out~H \times \ldots \times Out~H)}_{r~copies}$ by permuting the coordinates. Let  
$$ \tilde{\nu}: \underbrace{(Aut~H \times \ldots \times Aut~H)}_{r~copies}\rtimes S_r \longrightarrow \underbrace{(Out~H \times \ldots \times Out~H)}_{r~copies} \rtimes S_r$$ be the homomorphism 
  defined by $\tilde{\nu}(x_1, x_2, \ldots , x_r, x_{r+1})=(x_1Inn~H, \ldots , $ $x_rInn~H,x_{r+1})$. We denote by $\beta$ the projection of $\underbrace{(Out~H \times \ldots \times Out~H)}_{r~copies} \rtimes S_r$ onto the  $(r+1)$-th factor $S_r$, which is obviously a homomorphism.
 \smallskip
\begin{lemma} \label{ll}
Let $H$ be a finite nonabelian simple group. Then $Out~H$ does not contain a subgroup
 isomorphic to the quaternion group $Q_8$ of order $8$.
\end{lemma}
\begin{proof}
If $H$ is isomorphic to either alternating group $Alt_n$ of degree $n$ or to a Sporadic simple group, then $|Out(H)| \leq 4$ (see \cite[2.17, 2.19, p.299]{suz} and \cite[Table 2.1C, p.20]{lie}), so the Lemma follows in this case. If $H$ is isomorphic to a finite simple group of Lie type, then the Lemma follows by \cite[Theorem 2.5.12, p.58]{gor}.       
\end{proof}
\smallskip
\begin{corollary} \label{l}
Let $H$ be a finite nonabelian simple group. Then for any $m \in \N$,
 ${\underbrace{Out~H \times \ldots \times Out~H}_{m ~copies}}$ does not contain a subgroup
 isomorphic to the quaternion group $Q_8$ of order $8$.
\end{corollary}
\smallskip

\begin{proof}
Assume ~that~
 $\alpha$ ~is~ an ~injective~ homomorphism ~from ~$Q_8$~ to ~
 ${\underbrace{Out~H \times \ldots \times Out~H}_{m ~copies}}$. 
  Let $u=(x_1,x_2, \ldots , x_m)$ denote an element of $\alpha (Q_8)$ of order $4$. 
 Then there is $t$ ($1\leq t\leq m$) such that 
 $x_t$ is of order $4$. Let $p_t$ denote the projection 
 of ${\underbrace{Out~H \times \ldots \times Out~H}_{m ~copies}}$
 onto the $t$-th factor. 
 Then $(p_t\circ \alpha)(Q_8)$ is a subgroup of $Out~H$ which contains an element of order $4$. 
 Since a homomorphic image of $Q_8$ containing an element of order $4$ is isomorphic to $Q_8$, 
 $ (p_t\circ \alpha)(Q_8)\cong Q_8$. By Lemma \ref{ll}, this is impossible. 
\end{proof} 
\smallskip
\begin{theorem}\label{th1}
 A finite nonsolvable group $G$ is JND-group if and only if there exists a finite nonabelian simple group $H$, a natural number $r$ and a Dedekind group $D \subseteq \underbrace{(Out~H \times \ldots \times Out~H)}_{r~copies} \rtimes S_r$ such that 

\n (i) the usual action of $\beta (D)$ on the set $\{1,2, \ldots ,r \}$ is free and transitive, 

\n and

\n (ii) $G \cong \tilde{\nu}^{-1}(D)$,

\n where all the notations have meaning described as after the Proposition \ref{d4} 

\n Further, $G$ is JND but not JNA if and only if $D$ is a nonabelian Dedekind group and $r$ is even.
\end{theorem}
\begin{proof}
Suppose that $G$ is a nonsolvable JND-group. By Corollary \ref{d2}, $G$ is monolithic. Let $K$ denote the monolith of $G$. Since $G$ is nonsolvable and $K$ is characteristically simple, by \cite[3.3.15 (ii), p.87]{rob}, there exists a finite nonabelian simple group $H$ and a natural number $r$ such that $K \cong \underbrace{(H \times \ldots \times H)}_{r~copies}$.

By Proposition \ref{d4}, $G$ is semisimple. We show that $K$ is the CR-radical of $G$. Let $N$ be the CR-radical of $G$ containing $K$. Then $N$ is semisimple \cite[Lemma, p.205]{kur}. Assume that $N \neq K$. Then there exists nontrival completely reducible normal subgroup $L$ of $N$ which is complement of $K$ in $N$. Now since $L \cong N/K$ and $G/K$ is a Dedekind group, $L$ is solvable \cite[5.3.7, p.143]{rob}. Further, since nontrival normal subgroup of a semisimple group is semisimple \cite[Lemma, p.205]{kur}, $L$ is also semisimple. This is a contradiction.

Now by \cite[3.3.18 (i), p.89]{rob}, there exists ~$G^{\star} ~\cong ~~G$ ~such ~that ~$\underbrace{(Inn~H \times \ldots \times Inn~H)}_{r~copies} \leq G^{\star} \leq \underbrace{(Aut~H \times \ldots \times Aut~H)}_{r~copies}\rtimes S_r$.
  We identify $G$ with $G^{\star}$ and $H$ with $Inn~H$. Thus $K$ is identified with $\underbrace{(Inn~H \times \ldots \times Inn~H)}_{r~copies}$. Take $D=G/K \subseteq \underbrace{(Out~H \times \ldots \times Out~H)}_{r~copies}\rtimes S_r$. Then $D$ is a Dedekind group and $G \cong \tilde{\nu}^{-1}(D)$. This proves (ii).
   
Next, we claim that $\beta(D)$ acts transitively on the set of symbols $\{1,2, \ldots ,$ $r\}$. Let $O$ denote an orbit of the natural action of $\beta(D)$ on $\{ 1,2, \ldots , r\}$. Consider the subgroup $M_O=\{(x_1,x_2, \ldots ,x_r,1)| x_i \in Inn~H ~\text{and}~ x_i=1 ~\text{if} ~i \not \in O \}$ of $\underbrace{(Aut~H \times \ldots \times Aut~H)}_{r~copies}\rtimes S_r$.
It is easy to observe that for each element of $G$, the $(r+1)$-th coordinate is an element of $\beta(D)$. This implies that $M_O$ is a normal subgroup of $G$ contained in $K$. But $K$ is the monolith of $G$, so $M_O=K$. This proves that $O=\{1,2,\ldots , r\}$.

Now, we show that the action of $Z(\beta (D))$ on $\{1,2, \ldots , r\}$ is free.  Suppose that an element $u$ of $Z(\beta (D))$ fixes a symbol $a$ under the natural action of $\beta(D)$ on $\{1,2, \ldots, r\}$. Clearly $u$ fixes each element of the orbit $\beta(D).a$ of $a$ which is $\{1,2, \ldots , r\}$. This implies $u=1$. So, no nontrivial element of $Z(\beta (D))$ will fix any symbol in $\{1,2, \ldots ,r\}$.

If $D$ is abelian, then $Z(\beta(D))=\beta(D)$ and so the action of $\beta(D)$ is free. If $D$ is nonabelian Dedekind group, then by the structure Theorem for Dedekind groups \cite[5.3.7, p.143]{rob}, there exists a nonnegative integer $t$ and an abelian group $A_{o}$ of odd order such that we can write $D = Q_8 \times (\Z_2)^t \times A_{o}$. Thus for any $x \in \beta(D)$, either $x\in Z(\beta (D))$ or $1 \neq x^2 \in Z(\beta (D))$. This implies that a noncentral element $x$ also does not fix any symbol of set $\{1,2, \ldots ,r\}$ (for then $1 \neq x^2 \in Z(\beta (D))$ will fix that symbol). Thus action of $\beta (D)$ on $\{1,2, \ldots ,r\}$ is free and transitive. In particular $r=|\beta (D)|$. This proves (i).

Now, assume that $G$ is JND but not JNA. Then by Corollary \ref{l}, $\beta (Q_8) \neq 1$ and so $2$ divides $|\beta (D)|=r$

Conversely,~ suppose ~that ~there ~exists ~a ~Dedekind ~group 
 $D \subseteq  {\underbrace{(Out~H \times \ldots \times Out~H)}_{r~copies}}\rtimes S_r$ for some $r \in \N$ and a nonabelian finite simple group $H$ such that, the usual action of $\beta (D)$ on the set $\{1,2, \ldots ,r\}$ is free and transitive. Let $G=\tilde{\nu}^{-1}(D)$. We will show that $G$ is a JND-group. By \cite[3.3.18 (ii), p.89]{rob}, ~$G$ ~is ~semisimple ~with ~CR-radical ~$K~=~ {\underbrace{(Inn~H ~\times ~\ldots ~\times ~Inn~H)}_{r~copies}}$~ and~ ${\underbrace{(Inn~H \times \ldots \times Inn~H)}_{r~copies}} \leq G \leq {\underbrace{(Aut~H \times \ldots \times Aut~H)}_{r~copies}}\rtimes S_r$. We will show that $K$ is the monolith of $G$.  Since $K=K^{(1)}$, $K$ is contained in all terms of the derived series of $G$. Further, since $G$ is semisimple, there is smallest nonnegative integer $n$ such that $G^{(n)}=G^{(n+i)}$ for all $i \in \N$. This implies that $G^{(n)}/K$ is a perfect group. But since $G^{(n)}/K$ is Dedekind and so solvable \cite[5.3.7, p.143]{rob}, $G^{(n)}=K$. Let $N$ be a nontrival normal subgroup of $G$.
Since a nontrivial normal subgroup of a semisimple group is semisimple \cite[Lemma, p.205]{kur} and a semisimple group has trivial center, $N\cap K \neq \{1\}$.
By \cite[Theorem 2, p.156]{mil}, $N\cap K =\underbrace{N_1 \times N_2 \times \ldots \times N_r}_{r~copies}$, where $N_i \trianglelefteq Inn~H$ and at least one $N_i \neq 1$. Now since $N_i=Inn~H$ and $\beta(D)$ acts transitively on $\underbrace{Inn~H \times \ldots \times Inn~H}_{r~copies}$, so $N\cap K = \underbrace{Inn~H \times \ldots \times Inn~H}_{r~copies} =K$, that is $K \subseteq N$. This proves that $K$ is the monolith of $G$. Thus $G$ is JND-group. Further, if $D$ is nonabelian Dedekind group, then $G$ is JND but not JNA.

 																   
\end{proof}
\begin{remark}
Let $G$ be finite just nonsolvable (JNS) (respectively just nonnilpotent (JNN)) group. Let $n$ be the smallest nonnegative integer such that $G^{(n)}=G^{(n+k)}$ (respectively $\gamma _n(G) =\gamma _{(n+k)}(G)$) for all $k \in \N$. Then it is easy to see that $G^{(n)}$ (respectively $\gamma_{(n+1)}(G)$) is the monolith of $G$.

The idea of the proof of the above theorem can be used to show that:

\n A finite nonsolvable group $G$ is JNS-group (respectively JNN-group) if and only if there exists a finite nonabelian simple group $H$, a natural number $r$ and a solvable (respectively nilpotent) group $D \subseteq \underbrace{(Out~H \times \ldots \times Out~H)}_{r~copies} \rtimes S_r$ such that 

\n (i) the usual action of $\beta (D)$ on the set $\{1,2, \ldots ,r \}$ is transitive, 

\n and

\n (ii) $G \cong \tilde{\nu}^{-1}(D)$,

\n where all the notations have meaning described as after the Proposition \ref{d4}.

\end{remark}
 
\smallskip
\n \textbf{Acknowledgement:}
We thank Professor Ramji Lal for suggesting the problem and for several 
stimulating discussions.


\begin{thebibliography}{}

\bibitem[1]{b1} R. Baer. Situation der untergruppen and struktur der gruppe. {\it S. B. Heidelberg Akad. Mat. Nat}. {\bf 2} (1933), 12-17.

\bibitem[2]{burn} W. Burnside. {\it Theory of groups of finite order} (Dover Publications, Inc., second edition, 1955).

\bibitem[3]{fs}  Silvana Franciosi, Francesco de Giovanni. Solvable groups with many nilpotent quotients. {\it Proc. Roy. Irish Acad. Sect. A}. (1) {\bf 89} (1989), 43-52.

\bibitem[4]{gor} Daniel Gorenstein, Richard Lyons, Ronald Solomon. {\it The classification of the finite simple groups}. Mathematical surveys and Monographs (3)(40) (Amer. Math. Soc., 1998)
 
\bibitem[5]{kur} A. G. Kurosh. {\it The theory of groups}, vol. 2 (Chelsea Publishing Company New York, N. Y., 1956).

\bibitem[6]{lie} M.W. Liebeck, C.E. Praeger and J. Saxl. {\it The maximal factorisations of the finite simple groups and their automorphism groups}. Mem. Amer. Math. Soc. 432 (American Mathematical Society, 1990).

\bibitem[7]{mil} Michael D. Miller. On the lattice of normal subgroups of a direct product. {\it Pacific J. Math.} (2) {\bf 60} (1975), 153-158.

\bibitem[8]{new1} M. F. Newman. On a class of metabelian groups. {\it Proc. London Math. Soc.} (3) {\bf 10} (1960), 354-364.

\bibitem[9]{new2} M. F. Newman. On a class of nilpotent groups. {\it Proc. London Math. Soc.} (3) {\bf 10} (1960), 365-375.

\bibitem[10]{rob} D. J. S. Robinson. {\it { A course in the theory of groups}} (Springer-Verlag, 1996).

\bibitem[11]{rob73} D. J. S. Robinsion. Groups whose homomorphic images have a transitive normality relation. {\it Trans. Amer. Math. Soc.} {\bf 176} (1973), 181-213.

\bibitem[12]{suz} Michio Suzuki. {\it Group Theory I} (Springer-Verlag, 1982).  
\end{thebibliography}
\end{document}